\documentclass[a4paper,reqno]{amsart}
\usepackage{amsmath, amsthm, amssymb, mathabx}
\usepackage{subcaption}
\usepackage{graphicx}
\usepackage[margin=30mm]{geometry}	\DeclareGraphicsExtensions{.png,.pdf,.jpg,.jpeg}
\usepackage[colorlinks=true,allcolors=blue!70!black,bookmarksopen=true]{hyperref}
\usepackage{color}
\usepackage{xcolor}
\usepackage[alphabetic,nobysame, initials]{amsrefs}
\usepackage[capitalize,noabbrev]{cleveref}

\numberwithin{equation}{section}

\usepackage{pb-diagram,pb-xy}
\usepackage[all]{xy}
\usepackage{xcolor}
\usepackage{caption}
\usepackage{subcaption}
\usepackage{enumitem}
\usepackage{pinlabel}
\usepackage{tikz}
\usetikzlibrary{
  cd,
  calc,
  positioning,
  arrows,
  decorations.pathreplacing,
  decorations.markings,
}
\usepackage{tikzsymbols}
\usepackage{multirow,arydshln}

\makeatletter
\def\@captionheadfont{}
\makeatother

\newcommand{\TT}{\mathcal{T}} %tower
\newcommand{\Z}{\mathbb{Z}}
\newcommand{\R}{\mathbb{R}}

\newcommand{\ol}{\overline}
\newcommand{\sm}{\setminus}
\newcommand{\wt}{\widetilde}

\newcommand{\imra}{\looparrowright}

\DeclareMathOperator{\pt}{pt}

\theoremstyle{plain}
\newtheorem{theorem}{Theorem}[section]
	\newtheorem{proposition}[theorem]{Proposition}
	\newtheorem{lemma}[theorem]{Lemma}

		\newtheorem{theoremalpha}{Theorem}
		
\makeatletter
\newtheorem*{rep@theorem}{\rep@title}
\newcommand{\newreptheorem}[2]{%
\newenvironment{rep#1}[1]{%
 \def\rep@title{#2 \ref{##1}}%
 \begin{rep@theorem}}%
 {\end{rep@theorem}}}
\makeatother
\newreptheorem{theorem}{Theorem}
\newreptheorem{lemma}{Lemma}
\newreptheorem{proposition}{Proposition}
\newreptheorem{corollary}{Corollary}

\theoremstyle{definition}
	\newtheorem{definition}[theorem]{Definition}

	\newtheorem{remark}[theorem]{Remark}
	
\newtheoremstyle{theorem-giventitle}
        {}{}              %%% space between body and thm
        {\itshape}                      %%% Thm body font
        {}                              %%% Indent amount (empty = no indent)
        {\bfseries}                     %%% Thm head font
        {.}                             %%% Punctuation after thm head
        { }                             %%% Space after thm head
        {\thmnote{\bfseries#3}}%%% Thm head spec
\theoremstyle{theorem-giventitle}
\newtheorem{theorem-named}{}

\begin{document}

\title[The disc embedding theorem and dual spheres]{The $4$-dimensional disc embedding theorem and dual spheres}

\author{Mark Powell}
\address{School of Mathematics and Statistics, University of Glasgow, United Kingdom}
\email{mark.powell@glasgow.ac.uk}

\author{Arunima Ray}
\address{Max-Planck-Institut f\"{u}r Mathematik, Vivatsgasse 7, 53111 Bonn, Germany}
\email{aruray@mpim-bonn.mpg.de }

\author{Peter Teichner}
\address{Max-Planck-Institut f\"{u}r Mathematik, Vivatsgasse 7, 53111 Bonn, Germany}
\email{teichner@mac.com }

%\subjclass[2010]{57N13, 57N35}
%\subjclass[2020]{57K40, 57N35}
\def\subjclassname{\textup{2020} Mathematics Subject Classification}
\expandafter\let\csname subjclassname@1991\endcsname=\subjclassname
%\expandafter\let\csname subjclassname@2000\endcsname=\subjclassname

\begin{abstract}
We modify the proof of the disc embedding theorem for $4$-manifolds, which appeared as Theorem~5.1A in the book ``Topology of 4-manifolds'' by Freedman and Quinn, in order to construct geometrically dual spheres. These were claimed in the statement but not constructed in the proof.
We also prove Proposition 1.6 from the Freedman-Quinn book regarding generic homotopies of discs or spheres in a 4-manifolds, which was not proven there.
\end{abstract}
\maketitle

\section{Introduction}

The \emph{disc embedding theorem}~\cite{Freedman:1984-1} combines work of Casson, Freedman, and Quinn. In a topological $4$-manifold with \emph{good} fundamental group, the theorem replaces an immersed disc with embedded boundary and a framed algebraically dual sphere by a locally flat embedded disc with the same boundary as the original disc. Consequences include topological $4$-dimensional surgery theory and the topological $5$-dimensional $s$-cobordism theorem, both for good fundamental groups.

Freedman's original proof was restricted to the simply connected case and used \emph{Casson handles}, built out of layers of (thickened) immersed discs. The full disc embedding theorem was first stated in the book ``Topology of $4$-manifolds'' by Freedman and Quinn~\cite{FQ}*{Theorem~5.1A}. This book expanded on Freedman's argument by using a generalisation of Casson handles built out of capped gropes, variously called a skyscraper~\cite{Freedman-notes}, a cope~\cite{Freedman:1984-1}, or a generalized infinite tower~\cite{FQ}.  The Freedman-Quinn capped grope approach is key to the proof of the disc embedding theorem for nontrivial fundamental groups.

The goal of this article is to modify part of the Freedman-Quinn proof of the disc embedding theorem, in order to fill a gap in the proofs of~\cite{FQ}*{Theorem~5.1A~and~Corollary~5.1B} related to geometrically dual spheres. Briefly, one needs algebraically dual spheres for the input of the disc embedding theorem, while for many applications one needs geometrically dual spheres in the output.  Here we say that two surfaces are \emph{geometrically dual} if they intersect transversely in precisely one point. The proof in~\cite{FQ} produces algebraically dual spheres in the output, rather than geometrically dual spheres.

Part II of Freedman-Quinn~\cite{FQ} is the principal reference for the tools required to do any nontrivial work with topological 4-manifolds. These include the annulus theorem~\cite{FQ}*{Theorem~8.1A},  smoothing away from a point~\cite{FQ}*{Theorem~8.2}, transversality~\cite{FQ}*{Theorem~9.5A}, and the existence and uniqueness of normal bundles for locally flat submanifolds~\cite{FQ}*{Section~9.3}.  Every theorem in Chapters 7, 8, and 9 of~\cite{FQ} relies on Theorem~5.1A, including the latter's assertion of geometrically dual spheres.  In turn the classification results of Chapters 10 and 11, and indeed all classification results proven since then, make essential use of these tools.
Thus it is important to have a complete proof of Theorem 5.1A of~\cite{FQ}, from the point of view of both the rest of that book and of much of the rest of the literature on topological 4-manifolds.

In the spirit of addressing foundational omissions from~\cite{FQ}, we also prove \cite{FQ}*{Proposition~1.6} on homotopies between immersed surfaces in a 4-manifold. This states that a homotopy between generic immersions of a surface in a 4-manifold is homotopic to a composition of homotopies, each of which is a regular homotopy or a cusp homotopy in some ball, or the inverse of a cusp homotopy. This important proposition was stated but not proven in the Freedman-Quinn book. We will discuss the details in Section~\ref{subsec:generic-immersions}.

One might wonder whether our concern about the existence of geometrically dual spheres is valid, and whether, for example, the algebraically dual spheres produced in the proof in~\cite{FQ} might suffice for the above mentioned applications. However, surgery on an embedded, framed sphere without a geometrically dual sphere might change the fundamental group of the ambient $4$-manifold, whereas geometrically dual spheres guarantee control over the fundamental group. Therefore, many applications hinge on providing geometric duals together with the embedded spheres. We also note that a $2$-sphere in a $4$-manifold with an algebraically dual sphere need not admit a geometrically dual sphere, a phenomenon not seen in higher dimensions due to the Whitney trick.  In dimension four, geometric duals cannot be found \textit{post hoc}. An example of a sphere in $S^2\times S^2$ with an algebraically dual sphere but no geometrically dual sphere, was produced in \citelist{\cite{Sato91}*{Section~3}\cite{Sato89}*{Example~4.1}} cf.~\cite{klugmiller2019}*{Figure~9}. In the construction one begins with a $2$-knot $\Sigma\subseteq S^4$ and performs surgery on a simple closed curve in $S^4\smallsetminus \Sigma$ which is homologous to the meridian. The result is an embedded sphere $S$ in $S^2\times S^2$. By a judicious choice of $\Sigma$, one may ensure that $\pi_1((S^2\times S^2)\smallsetminus S)$ is a nontrivial perfect group, implying that $S$ does not admit a geometrically dual sphere.

\subsection{The disc embedding theorem}

The \emph{equivariant intersection form} of a connected 4-manifold $M$  is a pairing
\[\lambda \colon H_2(M,\partial M; \Z[\pi_1 M]) \times H_2(M; \Z[\pi_1 M]) \longrightarrow \Z[\pi_1 M].
\]
As well as having nonempty boundary, the topological $4$-manifold $M$ may be nonorientable or noncompact. For topological 4-manifolds, we will explain the \emph{reduced self-intersection number}
\[
\wt{\mu} \colon \pi_2(M) \longrightarrow \Z[\pi_1 M]/ \langle g - w(g)g^{-1},\Z\cdot 1 \rangle
\]
in \cref{subsec:generic-immersions}.
If $M$ is orientable, the quantity $\wt\mu$ is determined by $\lambda$. In this case, the assumptions~$\wt\mu(g_i)=0$ below are implied by $\lambda(g_i,g_i)=0$.

Given maps $\{f_i\}_{i=1}^k$ of discs or spheres to $M$, a second collection $\{g_i \colon S^2\to M\}_{i=1}^k$ is \emph{algebraically dual} to the $\{f_i\}_{i=1}^k$ if the algebraic intersection form satisfies $\lambda(f_i,g_j)=\delta_{ij}$.
If this is true geometrically,  we say that the collection $\{g_i\}$ is \emph{geometrically dual} to the $\{f_i\}$.
More precisely, this means that $f_i\cap g_j$ is a single, transverse point for $i=j$ and is empty for $i \neq j$. An intersection point between $f_i$ and $g_i$ is said to be \emph{transverse} if there are  local coordinates $\R^4 \hookrightarrow M$ in which $f_i$ and $g_i$ are linear.
Here, and throughout the paper, we conflate the maps $f_i$, $g_i$ with their images $f_i(S^2)$/$f_i(D^2)$ and $g_i(S^2)$, as well as with the corresponding homology/homotopy classes.

We remark that~\cite{FQ} uses the terminology of `algebraically transverse spheres' which by definition have trivial normal bundles. We make the condition on normal bundles more explicit in our nomenclature.
Our algebraically dual spheres need not have trivial normal bundles. When the latter condition is required, we state it explicitly.

\begin{theoremalpha}[Disc embedding theorem cf.\ \cite{FQ}*{Theorem~5.1A}]\label{thm:FQ51A-intro}
Let $M$ be a connected $4$-manifold with good fundamental group. Consider a continuous map
\[
F=(f_1,\dots,f_k) \colon (D^2\sqcup \cdots \sqcup D^2,S^1 \sqcup \cdots \sqcup S^1) \longrightarrow (M, \partial M)
\]
that is a locally flat embedding on the boundary and that admits algebraically dual spheres $\{g_i\}_{i=1}^k$ satisfying $\lambda(g_i,g_j)=0= \wt{\mu}(g_i)$ for all $i,j$.
Then there exists a locally flat embedding
\[
\ol{F}=(\ol{f}_1,\dots,\ol{f}_k) \colon (D^2\sqcup \cdots \sqcup D^2,S^1 \sqcup \cdots \sqcup S^1) \hookrightarrow (M,\partial M)
\]
such that $\ol{F}$ has the same boundary as $F$ and admits a generically immersed, geometrically dual collection of framed spheres~$\{\ol g_i\}_{i=1}^k$, such that $\ol{g}_i$ is homotopic to $g_i$ for each~$i$.

Moreover, if $f_i$ is a generic immersion, then it induces a framing of the normal bundle of its boundary circle. The embedding $\ol{f}_i$ may then be assumed to induce the same framing.
\end{theoremalpha}

\begin{remark}
Our statement of \cref{thm:FQ51A-intro} differs from~\cite{FQ}*{Theorem~5.1A} in that we emphasise that the input is purely homotopy theoretic, and we control the homotopy classes of the dual spheres. The interested reader can use the discussion in \cref{remark:normal-bundles} to see that even the original continuous maps $f_i$ induce framings modulo $2$ on the boundary circles.
\end{remark}

\begin{remark}
The geometrically dual spheres in the conclusion of the disc embedding theorem imply that inclusion induces an isomorphism $\pi_1(M) \cong \pi_1(M \sm \bigcup_i \ol{f}_i)$.
As in \cite{FQ}, we refer to a collection of immersed surfaces whose removal does not change the fundamental group as \emph{$\pi_1$-negligible}.
\end{remark}

We discuss topological immersions in \cref{subsec:generic-immersions},
where we also explain why a topological generic immersion $f$ admits a linear normal bundle $\nu(f)$ whose total space is embedded in~$M$ apart from finitely many plumbings. If the base of $\nu(f)$ is a disc, and so in particular contractible, then this bundle has a unique trivialisation, or framing, which is used in the last paragraph of the statement of \cref{thm:FQ51A-intro}.

We recall the notion of a \emph{good} group, from \cites{Freedman-Teichner:1995-1,DET-book-goodgroups}, in  \cref{def:good-group}.  For applications, it suffices to know that the class of good groups is known to contain groups of subexponential growth~\cites{Freedman-Teichner:1995-1, Krushkal-Quinn:2000-1}, and to be closed under subgroups, quotients, extensions, and colimits~\cite{FQ}*{p.\ 44}. In particular, all finite groups and all solvable groups are good.
It is not known whether non-abelian free groups are good.

The proof of the disc embedding theorem from~\cite{FQ} begins with immersed discs and has three distinct steps. We describe the $k=1$ case for ease of exposition. First use~\cite{FQ}*{pp.~86-7, Proposition~2.9 and Lemma~3.3} to upgrade the immersed disc $f$ to a \emph{$1$-storey capped tower}~$\mathcal{T}$ with at least four surface stages, whose attaching region coincides with the framed boundary of $f$. Then use~\cite{FQ}*{Proposition~3.8} to show that every $1$-storey capped tower with at least four surface stages contains a skyscraper with the same attaching region. Finally, using \emph{decomposition space theory}, \cite{FQ}*{Theorem 4.1} shows that every skyscraper is homeomorphic to a handle $D^2\times D^2$ relative to its attaching region.
The proof given in~\cite{FQ}*{pp.~86-7} does not mention geometrically dual spheres, although they appear in the theorem statement.  However, the claimed geometrically dual spheres in the output are used multiple times in~\cite{FQ}, as we indicate in \cref{sec:compendium}.

In this article we give a modified proof of the disc embedding theorem, including the construction of the claimed geometrically dual spheres. Specifically, we only modify the first step of the proof, producing $1$-storey capped towers equipped with geometrically dual spheres. Our modification utilises dual \emph{capped surfaces} obtained from Clifford tori. Such capped surfaces allow us to create arbitrarily many collections of pairwise disjoint geometrically dual spheres. The key insight in our argument is that the Clifford tori are close to the corresponding double points, which allows us to ensure that the tracks of certain null homotopies do not intersect them. This small amount of extra disjointness is just enough to make the argument work. In \cref{rem:referee-arguments} we also explain an alternative argument suggested by a referee.

\begin{remark}\label{remark:for-experts}
For experts, we explain the problem with the proof in \cite{FQ} in more detail. That proof begins with immersed discs with geometrically transverse spheres and upgrades the former to capped gropes with arbitrarily many surface stages, still equipped with geometrically transverse spheres $\{g_i\}$. The key step is finding a further stage of caps, constructing a $1$-storey capped tower. The tower caps come from null homotopies for the double point loops in the caps of the capped grope, which exist due to the assumption on fundamental groups. The null homotopies cannot be easily controlled, and therefore the tower caps may intersect the dual spheres $\{g_i\}$, as well as the lower stages of the tower, arbitrarily. The final step in the proof of~\cite{FQ}*{Lemma~3.3} is to push intersections of the tower caps with the grope caps or surfaces stages down to the base stage and tube into the spheres $\{g_i\}$. One sees then that the intersections between $\{g_i\}$ and the tower caps have not been controlled. These could be pushed down into the surface stages, but in that case the $1$-storey capped towers would still only have algebraically dual spheres.

It seems to have been tacitly assumed in \cite{FQ}*{Proposition~3.3~and~pp.~86--7} that \cite{FQ}*{Lemma~3.3} provides geometrically dual spheres.  In fact, the first paragraph of the proof on page 86 does not mention how to construct transverse spheres, and Lemma 3.3 does not claim to provide them.  So \cite{FQ} gives a correct proof of Theorem 5.1A without the last four words `and with transverse spheres' (known as geometrically dual spheres in our terminology).  On the other hand, as discussed above, it is assumed throughout Part II of \cite{FQ} that Theorem 5.1A constructs geometrically dual spheres.
\end{remark}

\subsection{Outline of the proof}
For experts, here is an outline of our argument -- more details can be found in \cref{sec:technical-lemma,sec:proofs}. The beginning follows that of~\cite{FQ}*{Theorem~5.1A}: assume the collections $\{f_i\}$ and $\{g_i\}$ are generically immersed, and after cusp homotopies that the intersections and self-intersections of $\{g_i\}$ are algebraically cancelling. Tube the $\{f_i\}$ into the $\{g_i\}$ to ensure that the intersections and self-intersections of $\{f_i\}$ are algebraically cancelling, and further arrange that $\{f_i\}$ and $\{g_i\}$ are geometrically dual by the geometric Casson lemma (\cref{lem:geometric-casson-lemma}).

Let $\{D_j\}$ denote framed, immersed Whitney discs pairing the intersections within and among the $\{f_i\}$. Clifford tori at the double points of $\{f_i\}$ are geometrically dual to $\{D_j\}$. These tori can be capped by meridional discs for $\{f_i\}$, tubed into $\{g_i\}$ so that they lie in the complement of $\{f_i\}$, so we have dual capped tori $\{T_j\}$ to the $\{D_j\}$.

Now comes our modification. Tube the intersections among $\{D_j\}$ into the $\{T_j\}$ to produce capped surfaces $\{D_j'\}$ with the same framed boundary. The caps of $\{D'_j\}$ and of the $\{T_j\}$ can be separated as follows. The intersections between the two sets of caps are paired by framed, immersed Whitney discs. These may be assumed to be disjoint from the Clifford tori since the tori are located in a tubular neighbourhood of $\{f_i\}$, and they can be made disjoint from the bodies of $\{D_j'\}$ by pushing down and tubing into geometrically dual spheres for $\{D'_j\}$ produced by contracting parallel copies of the $\{T_j\}$. Now the geometric Casson lemma, applied to these new Whitney discs, separates the two families of caps, without creating any new undesirable intersections. Next we similarly upgrade the caps of $\{D'_j\}$ to capped surfaces, with caps disjoint from those of the $\{T_j\}$. To do so, push the cap intersections for $\{D'_j\}$ down and tube into $\{T_j\}$ to obtain height two capped gropes pairing the intersections among $\{f_i\}$. Then separate the caps of these height two capped gropes from the caps of the $\{T_j\}$ using the separation argument above, so that the final height two capped gropes are geometrically dual to the $\{T_j\}$.

The remainder of the proof is standard. By grope height raising and the good group hypothesis replace the height two capped gropes by height four capped gropes with null-homotopic double point loops and the same framed boundary. Let $\{C_\ell\}$ denote immersed discs bounded by these double point loops, arising from the null homotopies and with the appropriate boundary framing. We may assume that $\{C_\ell\}$ is disjoint from the Clifford tori (not their caps), since these lie in a tubular neighbourhood of $\{f_i\}$. For any intersections of $\{C_\ell\}$ with the rest of the height four gropes, push down and tube into geometrically dual spheres produced by contracting parallel copies of $\{T_j\}$. Now the height four capped gropes can be equipped with $\{C_\ell\}$ as tower caps to produce 1-storey capped towers $\{\TT^c_j\}$. Finally, the $\{T_j\}$ are contracted to spheres $\{R_j\}$ which are geometrically dual to $\{\TT^c_j\}$. By tubing $\{g_i\}$ into $\{R_j\}$, we acquire spheres $\{\ol{g}_i\}$ that are geometrically dual to $\{f_i\}$ and do not intersect $\{\TT^c_j\}$. Freedman and Quinn~\cite{FQ}*{Chapters~3~and~4} (see also \cite{Freedman-notes}*{Parts~II~and~IV}) showed that $1$-storey capped towers with at least four surface stages contain embedded topological discs with the same framed boundary. These can be used to perform the Whitney move on $\{f_i\}$ to produce the desired embeddings $\{\ol{f}_i\}$ with geometric duals $\{\ol{g}_i\}$.

Finally, in order to see that each $\ol{g}_i$ is homotopic to $g_i$, we note that the latter only changes by homotopy and tubing into the spheres $\{R_j\}$. But each sphere $R_j$ is null-homotopic in $M$ since it is constructed by contracting a Clifford torus (\cref{lem:null-homotopic}).

\begin{remark}\label{rem:referee-arguments}
We now describe an alternative method to obtain the geometrically dual spheres in \cref{thm:FQ51A-intro} suggested by the referee.

As in the outline above, we can arrange that $\{f_i\}$ and $\{g_i\}$ are generically immersed, that the intersections and self-intersections of both $\{f_i\}$ and of $\{g_i\}$ are algebraically cancelling, and further that $\{f_i\}$ and $\{g_j\}$ are geometrically dual. Therefore the intersections and self-intersections of $\{f_i\}$ are paired by framed, immersed Whitney discs. By tubing $\{f_i\}$ to itself along one of each pair of Whitney arcs, the collection is upgraded to capped surfaces $\{f'_i\}$. Meridional discs and the Whitney discs provide the caps, after tubing into $\{g_i\}$ to ensure disjointness from the bodies of~$\{f'_i\}$.

Next, repeat the argument for the caps, as follows. First push all the intersections of $\{g_i\}$ with the caps of $\{f'_i\}$ down to the body and then into the (unpushed) $\{g_i\}$. This preserves the algebraic intersection conditions on $\{g_i\}$, while ensuring that the $\{g_i\}$ no longer intersect the caps of $\{f'_i\}$. Now the intersections among the caps of $\{f'_i\}$ are pushed down and tubed into the $\{g_i\}$, to arrange that the intersections among the caps of $\{f'_i\}$ are algebraically cancelling. Then we can assume that the cap intersections are paired by framed, immersed Whitney discs. Tube along one of each pair of Whitney arcs, and cap with meridional and Whitney discs, with intersections with $\{f'_i\}$ pushed down and tubed into $\{g_i\}$ as needed, to upgrade $\{f'_i\}$ to a collection of height two capped gropes.

Now apply grope height raising and the good group hypothesis to this collection to produce height four capped gropes with the same framed attaching region as $\{f_i\}$ and null-homotopic double point loops.

Let $\{C_\ell\}$ denote immersed discs bounded by these double point loops, arising from the null homotopies and with the appropriate boundary framing. For every intersection of $\{C_\ell\}$ with the height four capped gropes, push down and tube into parallel copies of $\{g_i\}$. Now the height four capped gropes can be equipped with $\{C_\ell\}$ as tower caps to produce 1-storey capped towers $\{\TT^c_i\}$. This is the outcome of \cite{FQ}*{Lemma~3.3}. Note that $\{g_i\}$ and the bodies of $\{\TT^c_i\}$ are geometrically dual, but we do not have any control yet on the intersections between $\{g_i\}$ and $\{C_\ell\}$.

Now we have the modification proposed by the referee. Take parallel push-offs of $\{g_i\}$, which we call $\{g_i'\}$. For every intersection of $\{C_\ell\}$ with $\{g_i\}$, push $\{g_i\}$ down to the base surface of the tower, and tube into copies of $\{g_i'\}$. Still call the result $\{g_i\}$. Note that we have now arranged that the intersections between $\{C_\ell\}$ and $\{g_i\}$ are algebraically cancelling, since by the pushing down procedure, each intersection between some element of $\{C_\ell\}$ and the original set of spheres $\{g_i\}$ has led to some even number of (algebraically cancelling) tubings into $\{g_i'\}$. Let $\{W_j\}$ denote a set of framed, immersed Whitney discs pairing the intersections between $\{C_\ell\}$ and~$\{g_i\}$. Remove any intersections between $\{W_j\}$ and anything in the 1-storey towers $\{\TT^c_i\}$ below the tower caps by pushing down and tubing into $\{g_i'\}$. Now use the geometric Casson lemma (\cref{lem:geometric-casson-lemma}) with $\{W_j\}$ to separate the caps of $\{\TT_i^c\}$ from $\{g_i\}$. The result is a new collection of 1-storey capped towers with geometrically dual spheres $\{g_i\}$.

The rest of the argument consists of first upgrading the $1$-storey capped towers to skyscrapers, also equipped with geometrically dual spheres, and then showing that the homeomorphism from any skyscraper to $D^2\times D^2$, relative to the attaching region, preserves the geometrically dual spheres. One can verify that the proofs of~\cite{FQ}*{Proposition~3.8~and~Theorem~4.1} accomplish this, by moving the intersection point with the dual spheres sufficiently close to the preserved attaching region.

Finally we check that the geometrically dual spheres produced are homotopic to the algebraically dual spheres in the hypotheses. Note that in the construction above each $g_i$ has been changed by homotopies and by tubing into $\{g'_i\}$. Since the tubing occurred after pushing down, each $g_j'$ is used for these tubing operations an algebraically cancelling number of times and therefore there is no  change in the homotopy class of $g_i$, as desired.
\end{remark}

\subsection{Generic immersions and intersection numbers}\label{subsec:generic-immersions}

Homotopy classes of smooth maps of a compact surface to a 4-manifold are represented by \emph{generic immersions}, which are immersions whose only singularities are transverse double points in the interior. In the topological category, we use this local description as the definition of a generic immersion. In particular, a generic immersion is locally a flat embedding and hence restricts to a locally flat embedding of the boundary. A \emph{regular homotopy} in the smooth category is a homotopy through immersions. A smooth regular homotopy of generically immersed surfaces in a $4$-manifold is generically a concatenation of (smooth) isotopies, finger moves, and Whitney moves~\cite{GoGu}*{Section~III.3}. A \emph{topological regular homotopy} of generically immersed surfaces in a $4$-manifold is by definition a concatenation of (topological) isotopies, finger moves, and Whitney moves.

In order to work effectively in a topological 4-manifold, it is key to be able to assume that a continuous map of a surface can be perturbed to a topological generic immersion~\cite{FQ}*{Lemma~1.2}. For instance, this is the first step in the proof of the disc embedding theorem in a topological 4-manifold. It also follows that a topological generic immersion $f\colon \Sigma \imra M$ has a linear normal bundle $\nu(f)$. Similarly, it is also essential to be able to decompose a topological homotopy into a sequence of regular homotopies and cusp homotopies~\cite{FQ}*{Proposition~1.6}. This latter proposition was stated in~\cite{FQ} without a proof, so we provide one.

Specifically, the combination of~\cite{FQ}*{Lemma~1.2~and~Proposition~1.6} can be stated in the following useful way, which is what we prove in~\cref{sec:generic-immersions}.

\begin{theorem}\label{thm:gen-immersions-intro}\label{theorem:generic-immersions-bijection}
Let $\Sigma$ be a disjoint union of discs or spheres, and let $M$ be a $4$-manifold.
  The subspace of generic immersions in the space of all continuous maps
leads to a bijection
\[
\frac{\{(f_1,\dots,f_m)\colon \Sigma= \Sigma_1  \sqcup \cdots \sqcup \Sigma_m  \imra M \mid \mu(f_i)_1 = 0,\, i=1,\dots,m\}} { \{\text{isotopies, finger moves, Whitney moves}\} } \longleftrightarrow [\Sigma,M]_{\partial}, \]
where $\mu(f_i)_1 \in\Z$ denotes the signed sum of double points of $f_i$ whose double point loops are trivial in $\pi_1(M)$, and $[\Sigma,M]_{\partial}$ denotes the set of homotopy classes of continuous maps that restrict on~$\partial\Sigma$ to locally flat embeddings disjoint from the image of the interior of $\Sigma$.
Moreover, any such homotopy between generic immersions is homotopic rel.\ $\Sigma \times \{0,1\}$ to a sequence of isotopies, finger moves and Whitney moves.
\end{theorem}

As usual, the self-intersection number $\mu(f)$ of a generic immersion $f$ of a disc or sphere into $M$ is obtained by summing signed group elements $g \in \pi_1(M)$ corresponding to double points in the interior of $f$~\citelist{\cite{Wall-surgery-book}*{Chapter~5}\cite{FQ}*{Section~1.7}}.
The ambiguity of the choice of sheets at each double point leads to the relations $g - w(g)g^{-1}$ in $\Z[\pi_1 M]$. Local cusp homotopies allow us to change the coefficient $\mu(f)_1$ at the trivial fundamental group element~$1$ at will, and correspondingly we factor out by $\Z\cdot 1$. We obtain the \emph{reduced} self-intersection number, already used in the assumptions of \cref{thm:FQ51A-intro}:
\[
\wt{\mu} \colon \pi_2(M) \longrightarrow \Z[\pi_1 M]/ \langle g - w(g)g^{-1},\Z\cdot 1 \rangle.
\]
While $\mu(f)$ depends on the choice of generically immersed representative $f\colon S^2 \looparrowright M$, the reduced version $\wt{\mu}$ only depends on $[f]\in\pi_2(M)$. Any two generic immersions homotopic to $f$ are of course themselves homotopic. By the injectivity in \cref{thm:gen-immersions-intro} we see that $\wt{\mu}$ is well-defined on $\pi_2(M)$, since the quantity is evidently unchanged by isotopies, finger moves, and Whitney moves.

For any $f\colon S^2\looparrowright M$, the invariant $\mu$ satisfies
\begin{equation}\label{qr}
\lambda(f,f) = \mu(f) + \overline{\mu(f)} + e(f)\cdot 1 \in \Z[\pi_1 M]
\end{equation}
where $e(f)\in \Z$ is the \emph{Euler number} of the normal bundle $\nu(f)$ and the involution $\ol{g}:=w(g)g^{-1}$ is extended linearly to the group ring. Apply the augmentation map $\varepsilon\colon\Z[\pi_1 M]\to\Z$ to this equation, to see that modulo~2, $\varepsilon(\lambda(f,f)) \equiv e(f)$ only depends on the homotopy class of $f$. This is also the \emph{Stiefel-Whitney number} $w_2(f)\in\Z/2$ of $\nu(f)$ which will be used in the assumption of \cref{thm:FQ51B} in the next section.  As in the discussion above one sees that $w_2(a)=0$ if and only if~$a\in\pi_2(M)$ is represented by a generic immersion $f\colon S^2\looparrowright M$ that can be framed, i.e.\ whose normal bundle is trivial.

\subsection*{Conventions}
 All manifolds are assumed to be based, in order to define homotopy groups and equivariant intersection numbers. Topological embeddings are always assumed to be locally flat.

\subsection*{Outline}

In \cref{sec:compendium} we give applications of the disc embedding theorem with geometrically dual spheres %from~\cite{FQ}*{Theorem~5.1A and Corollary 5.1B}
from the literature. In particular we recall the sphere embedding theorem, and the \emph{hyperbolic embedding theorem}, which gives us the ability to represent a rank $2k$ hyperbolic summand of the intersection form by an $\#_{i=1}^k S^2 \times S^2$ connected summand.
In \cref{sec:prelims}, we review the objects, geometric constructions, and further results needed for the proof of the disc embedding theorem, then \cref{sec:technical-lemma} contains the main technical results needed for the construction of geometrically dual spheres. These results are applied in \cref{sec:proofs} to prove the disc embedding theorem (\cref{thm:FQ51A-intro}). Here we also explain an alternative argument kindly suggested by a referee.  Finally, in \cref{sec:generic-immersions} we discuss generic immersions and prove \cref{thm:gen-immersions-intro}.

\subsection*{Acknowledgements}
We are indebted to the $4$-manifolds semester at the Max Planck Institute for Mathematics in Bonn in 2013, and in particular to lectures of Michael Freedman, Frank Quinn, and Robert Edwards given as part of the semester. Much of the work leading to this paper took place at the Hausdorff Institute for Mathematics in 2016, or during subsequent visits by MP to the MPIM, or by AR to Durham University. We thank these institutions for their hospitality.
We thank Stefan Friedl, Daniel Kasprowski, Slava Krushkal, Allison N.~ Miller, Patrick Orson, and Frank Quinn for their insights and valuable discussions. We thank Michael Freedman for encouraging us to write this paper. Finally we are grateful to an anonymous referee for comments that helped improve the exposition, and the alternative argument outlined in \cref{rem:referee-arguments}.

MP was partially supported by EPSRC New Investigator grant EP/T028335/2 and EPSRC New Horizons grant EP/V04821X/2.

\section{Applications of geometrically dual spheres in the literature}\label{sec:compendium}

\subsection{The sphere embedding theorem}\label{sec:sphere-embedding-intro}
A central application of the disc embedding theorem is to prove the \emph{sphere embedding theorem}. The existence and exactness of the topological surgery sequence in dimension four are proven by changing a collection of generically immersed spheres with vanishing intersection and self-intersection numbers, by a regular homotopy, to pairwise disjoint embedded spheres~\cite{FQ}*{Theorem~11.3A}. The sphere embedding theorem describes precisely when such embedded spheres may be found. The proof of~\cite{FQ}*{Theorem~11.3A} neglects to mention that geometrically dual spheres are essential for performing surgery on $4$-manifolds without inadvertently modifying the fundamental group. The sphere embedding theorem is also integral to any known classification result for topological $4$-manifolds, including those that use Kreck's modified surgery theory~\cite{surgeryandduality}, for example \cite{HKT}. Another example is the main step in \cite{Hambleton-Kreck:1988-1}*{Lemma~4.1}, which uses sphere embedding to move from the easier stable classification up to connected sums with copies of $S^2 \times S^2$ to unstable classification results.
We refer to \cite{DET-book-s-cob-SET}*{Section~20.3} for a proof of the sphere embedding theorem.

\begin{theoremalpha}[Sphere embedding theorem with framed duals]\label{sphere-embedding-thm}\label{thm:FQ51B}
Let $M$ be a connected $4$-manifold with good fundamental group and consider a continuous map
\[
F = (f_1,\dots,f_k)\colon (S^2 \sqcup \cdots \sqcup S^2) \longrightarrow M
\]
satisfying $\wt\mu(f_i)=0$ for every $i$ and $\lambda(f_i,f_j)=0$ for $i\neq j$, with a collection of algebraically dual spheres $\{g_i\}_{i=1}^k$ with $w_2(g_i)=0$ for each $i$.
Then there is a locally flat embedding
\[
\ol{F} = (\ol{f}_1,\dots,\ol{f}_k)\colon (S^2 \sqcup \cdots \sqcup S^2) \hookrightarrow M
\]
with $\ol{F}$ homotopic to $F$ and with a generically immersed, geometrically dual collection of framed spheres $\{\ol{g}_i\}_{i=1}^k$, such that $\ol{g}_i$ is homotopic to $g_i$ for each~$i$.

Moreover, if $f_i$ is a generic immersion and $e(f_i)\in\Z$ is the Euler number of the normal bundle~$\nu(f_i)$, then $\ol{f}_i$ is regularly homotopic to $f_i$ if and only if $e(f_i)=\lambda(f_i,f_i)$.
\end{theoremalpha}

\begin{remark}
The assumption  $\wt\mu(f_i)=0$ implies that $\lambda(f_i,f_i) \in \Z\cdot 1 \subseteq \Z[\pi_1 M]$, by~\eqref{qr}.
In this case, \eqref{qr} gives an equation of integers $\lambda(f_i,f_i) = 2\cdot \mu(f_i)_1 + e(f_i)$ and hence the last condition~$e(f_i)=\lambda(f_i,f_i)$ in \cref{sphere-embedding-thm} is equivalent to the vanishing of $\mu(f_i)$.
\end{remark}

\begin{remark}\label{rem:g-Lagrangian}
If we assume in addition to the hypotheses of \cref{sphere-embedding-thm} that $\lambda(f_i,f_i)=0$ for all $i$, then we can get disjointly embedded framed spheres $\{\ol{f}_i\}$ as an output, since then $w_2(\ol{f}_i)=w_2(f_i)\equiv \varepsilon(\lambda(f_i,f_i))=0$ for all $i$, by the equality~\eqref{qr}. As a consequence, we can surger $M$ along~$\{\ol{f}_i\}$ to obtain a 4-manifold $M'$. The existence of the geometric duals $\ol{g}_i$ implies that $\pi_1(M') \cong \pi_1(M)$, but indeed more is true: each intersection (or self-intersection) of some $\ol{g}_j$ with a fixed $\ol{g}_i$ can be tubed into the unique intersection point between $\ol{g}_i$ and (a parallel copy of)~$\ol{f}_i$, resulting in geometric duals $h_i$ to $\ol{f}_i$ that are now disjointly embedded. This means that each pair $(\ol{f}_i, h_i)$ has a regular neighbourhood that is a sphere bundle over $S^2$, with a 4-ball removed. The sphere bundle is trivial if and only if $w_2(h_i)=w_2(g_i)=0$.

So in the setting of \cref{sphere-embedding-thm} where all $w_2(g_i)$ are assumed to vanish, we get a connected sum decomposition $M$ as a connected sum with copies of $S^2\times S^2$. Note that the $a_i:=[f_i]$ form the first Lagrangian, but the second Lagrangian generated by $b_i:=[h_i]$ is formed by linear combinations of the $[g_i]$ and $[f_i]$ induced by the geometric manoeuvres above.
\end{remark}

We record an important  special case of the situation described in \cref{rem:g-Lagrangian} in the following theorem.

\begin{theorem}[Hyperbolic embedding theorem]\label{cor:hyperbolic}
Let $M$ be a connected 4-manifold with good fundamental group and let $H$ be a hyperbolic form in $(\pi_2(M), \lambda_M, \wt\mu_M)$, meaning that $H$ is a $\Z[\pi_1 (M)]$-submodule of $\pi_2(M)$, generated by a hyperbolic basis consisting of classes
$a_1,\dots, a_k, b_1, \dots, b_k\in H$ with
\[
\lambda(a_i,b_j)=\delta_{ij},\, \lambda(a_i, a_j)=0=\lambda(b_i, b_j) \text{ and }  \wt\mu(a_i)=0=\wt\mu(b_i) \text{ for all }  i,j.
\]
Then there is a homeomorphism $M \approx (\#_{i=1}^k S^2 \times S^2) \# M'$ with a connected sum that on $\pi_2$ sends $a_i$ to $[S^2_i \times \{\pt_i\}]$ and $b_i$ to $ [\{\pt_i\} \times S^2_i]$. In particular, $H$ is an orthogonal summand freely generated by $\{a_i,b_i\}$ and $\pi_2(M) \cong H \perp \pi_2(M')$.
\end{theorem}

\cref{cor:hyperbolic} was stated on the first page of Freedman's ICM talk \cite{Freedman:1984-1}. Together with Donaldson's theorem on definite intersection forms for smooth 4-manifolds~\cite{Donaldson}, it implies the existence of infinitely many non-smoothable 4-manifolds. For example, a simply connected, closed $4$-manifold with intersection form $E_8 \oplus E_8$ can be obtained from the $K3$ surface by removing a hyperbolic form of rank~$6$. Since the form $E_8 \oplus E_8$ is definite but not diagonalisable, it cannot be realised by a closed, smooth, 4-manifold.

\begin{proof}[Proof of \cref{cor:hyperbolic}]
Represent the $a_i$ and $b_i$ by framed generic immersions, using \cref{prop:hom-gen-immersion} and the fact that in $\Z/2$ we have $w_2(a_i) \equiv \varepsilon(\lambda(a_i,a_i))=0$ and $w_2(b_i) \equiv \varepsilon(\lambda(b_i,b_i)) = 0$ for all $i$. Apply \cref{sphere-embedding-thm} to these framed generic immersions. The output is as in \cref{rem:g-Lagrangian}, except that we have the additional information that $\lambda(b_i,b_j)=0=\wt\mu(b_i)$ for all $i,j$. This means that the operations of tubing into the $\{\ol{f}_i\}$ occur in such a way that up to homotopy there is no effect, i.e.\ the resulting disjointly embedded spheres $\{h_i\}$ still represent the~$\{b_i\}$ in $\pi_2(M)$.
\end{proof}

\begin{remark}\label{rem:easier-SET}
In this paper we prove the disc embedding theorem with geometrically dual spheres, and then use it to deduce the sphere embedding theorem with geometrically dual spheres. Since the former is used throughout~\cite{FQ}, our proof fills the gap in that source. However, for applications in surgery theory, one begins with a dual spherical sublagrangian of the intersection form, as in \cref{cor:hyperbolic}.
In this case, the following proof of \cref{cor:hyperbolic} is available, which applies the disc embedding theorem without geometrically dual spheres. The idea is due to Casson~\cite{Casson}; we thank Slava Krushkal for reminding us of it.

Suppose we are given classes $a_1,\dots, a_k, b_1, \dots, b_k\in \pi_2(M)$ for a topological $4$-manifold $M$ with good fundamental group satisfying
\[
\lambda(a_i,b_j)=\delta_{ij},\, \lambda(a_i, a_j)=0=\lambda(b_i, b_j) \text{ and }  \wt\mu(a_i)=0=\wt\mu(b_i) \text{ for all }  i,j.
\]
Assume $\{a_i\}$ and $\{b_i\}$ are represented by immersed spheres $\{A_i\}$ and $\{B_i\}$ respectively. Let $\{A_i^+\}$ and $\{B_i^+\}$ denote push-offs of $\{A_i\}$ and $\{B_i\}$ respectively. For each $i$, choose small $4$-discs around the unpaired intersection point between $A_i$ and $B_i$, not intersecting $\{A_i^+\}$ and $\{B_i^+\}$. Tube the discs together using embedded $1$-handles in $M$ without intersecting $\{A_i\}\cup \{A_i^+\}\cup\{B_i\}\cup\{B_i^+\}$, and call the result $D$. The portion of $\{A_i\}$ and $\{B_i\}$ lying in $M\sm D$ are immersed discs $\{A_i^0\}$ and $\{B_i^0\}$ with trivial intersection and self-intersection numbers,  equipped with algebraically dual spheres $\{B_i^+\}\cup \{A_i^+\}$. Apply local cusp moves to $\{A_i^0\}$ and $\{B_i^0\}$ to ensure that the unreduced self-intersection numbers vanish. Apply the version of~\cite{FQ}*{Corollary~5.1B} without geometrically dual spheres to $\{A_i^0\}$ and $\{B_i^0\}$, to produce regularly homotopic, disjoint, embedded discs with the same framed boundary. Replacing the portions of $\{A_i\}$ and $\{B_i\}$ within $D$ produces the desired geometrically transverse embeddings in the classes $a_1,\dots, a_k, b_1, \dots, b_k\in \pi_2(M)$.

It is tempting to attempt a similar strategy to prove the general disc and sphere embedding theorems with geometrically dual spheres. However, this does not work, even for sphere embedding. In the general sphere embedding scenario, one set of spheres may not have the requisite triviality of intersection and self-intersection numbers. This arises in applications, for example in the construction of star partners for $4$-manifolds with odd intersection form~\cite{FQ}*{Section~10.4}, \cite{Stong-conn-sum} \cite{Teichner:1997-1}, \cite{Kasprowski-Powell-Ray-survey}.  Also, it was essential that we were working with spheres, since `dual discs' are not as helpful.
\end{remark}

\subsection{Other instances of geometrically dual spheres}

First we consider the book~\cite{FQ}, which is generally regarded as the canonical source for the ramifications of the disc embedding theorem.

\begin{enumerate}
    \item On page 105 of \cite{FQ}, in Part II of the book, within the proof of the technical version of $h$-cobordism theorem, it is claimed that $1$-storey capped towers with geometrically dual spheres were constructed in Part I. Geometrically dual spheres are used on page 107, where they are key to proving the \emph{negligibility property}.
\item The technical version of the $h$-cobordism theorem is used in the proof of the technical controlled $h$-cobordism theorem (\cite{FQ}*{Theorem~7.2C}), which in turn
appears in the proofs  of the proper $h$-cobordism theorem (Corollary~7.3C) and the annulus conjecture (Theorem~8.1A).
The negligibility property (called a \emph{regular homotopy property} in the proof of 8.1A) is invoked in an essential way.
\item The annulus conjecture is used in \cite{FQ} to prove topological transversality (Section~9.5), existence and uniqueness of normal bundles (Section~9.3), and smoothing results (Chapter~8).
\item The geometrically dual spheres claimed in \cref{thm:FQ51B} (\cite{FQ}*{Corollary~5.1B}) also arise in \cite{FQ} in the proof of the $\pi_1$-negligible embedding theorem (Theorem~10.5A), the plus construction (Theorem~11.1A), and the $\pi-\pi$ lemma (p.\ 216).
\end{enumerate}

\noindent Next we consider geometrically dual spheres elsewhere in the literature.
\begin{enumerate}[resume]
\item Similar to classical surgery theory, uses of Kreck's modified surgery~\cite{surgeryandduality} to obtain classification results on $4$-manifolds, for example in~\cite{Hambleton-Kreck:1988-1}, \cite{Hambleton-Kreck-93} and \cite{HKT}, needs geometrically dual spheres to avoid changing the fundamental group by surgery on embedded  spheres obtained using the sphere embedding theorem.
\item
For $M$ a compact 4-manifold with $\pi_1(M)$ good, every element of the Whitehead group $\operatorname{Wh}(\pi_1(M))$ is realised as the Whitehead torsion of an $h$-cobordism $(W;M,M')$ based on~$M$. To prove this, one builds a cobordism with 2- and 3-handles. When attaching the 3-handles, one can find smoothly embedded spheres in the desired homotopy classes, but it is not known how to find such embeddings smoothly that also come with geometrically dual immersed spheres. Without the geometric duals, the inclusion induced map~$\pi_1(M') \to \pi_1(W)$ need not be an isomorphism.  Thus the sphere embedding theorem is needed to obtain topologically locally flat embeddings for the 3-handle attachments, with geometrically dual spheres. See~\cite{Kasprowski-Powell-Ray-survey}*{Theorem~3.5}.
\item
Geometrically dual spheres are needed to show that the complement of the topological slice disc produced in~\cite{Garoufalidis-Teichner} for a knot with Alexander polynomial one has fundamental group~$\Z$, and similarly in \cite{FriedlTeichner} for slice disc exteriors with fundamental group $\Z \ltimes \Z[1/2]$.
\end{enumerate}

The papers of Casson~\cite{Casson} and Freedman~\cites{F,Freedman:1984-1} do not have problems relating to geometrically dual spheres.
A minor point is that in~\cite{F}*{Theorem~1.2}, Freedman claimed to prove the exactness of the surgery sequence by embedding half of each hyperbolic pair representing the surgery kernel in a simply connected manifold by an embedded sphere, however this can be easily fixed using the method of \cref{rem:easier-SET} instead.

\section{Definitions and operations}\label{sec:prelims}

We shall assume that a Whitney disc is a generic immersion $W\colon D^2\imra M$, which in particular implies that the boundary is embedded, with each Whitney arc lying on one of the sheets whose intersection points are paired by $W$.  We allow interior self-intersections of $W$ as well as intersections with other surfaces (and other Whitney discs) but assume that these are transverse, using topological transversality. As usual, a Whitney disc $W$ is said to be \emph{framed} if the Whitney section of the normal bundle $\nu(W)$ of $W$ restricted to its boundary extends to a non-vanishing section of~$\nu(W)$. The relative normal Euler number of $W$ in $\Z$ is by definition the unique obstruction for the existence of a framing of $\nu(W)$ extending the Whitney section.
The framing of a Whitney disc can be altered by interior twists and boundary twists; see~\cite{FQ}*{Section~1.3}.

Given two Whitney discs, the corresponding Whitney circles may \textit{a priori} intersect one another.
We can ensure that Whitney circles are disjoint by pushing one Whitney circle along the other.
From now on, we will assume that Whitney circles are pairwise disjoint and embedded.

In our proofs, we will construct \emph{gropes} and ultimately \emph{towers}, as defined in~\cite{DET-book-gropestowers} (see also~\cite{FQ}*{Chapters~2~and~3}). We assume the reader is familiar with these definitions. We will need the following generalisation of the notion of dual spheres.

\begin{definition}[Dual capped gropes and surfaces]\label{defn:dual-cappped-gropes-surfaces}
Let $\{A_i\}$ be a collection of immersed discs, gropes, or towers, with or without caps. A collection $\{T_i^c\}$ of (sphere-like) capped gropes is said to be (geometrically) \emph{dual} to $\{A_i\}$ if $T_i^c\pitchfork A_j$ is a single transverse point when $i=j$, located in the bottom stages of $T_i^c$ and $A_i$, and $T_i^c\pitchfork A_j$ is empty otherwise. For $\{A_i\}$ a generically immersed collection of discs, the \emph{bottom stage} of $A_i$ is simply itself. Note that all the caps of $\{A_i\}$ are required to be disjoint from the caps of $\{T_i^c\}$. Additionally, note that intersections are allowed within the collection~$\{T_i^c\}$.
If each $T_i^c$ is a capped grope of height one, then we say that $\{T_i^c\}$ is a collection of \emph{dual capped surfaces} for $\{A_i\}$.
\end{definition}

We will use the definition of a good group in the proof of \cref{lem:grope-to-tower}, so we recall it here.

\begin{definition}[\cites{Freedman-Teichner:1995-1,DET-book-goodgroups}]\label{def:good-group}
    A group $\Gamma$ is said to be \emph{good} if for every height $1.5$ disc-like capped grope $G^c$, with some choice of basepoint, and for every group homomorphism $\phi\colon \pi_1(G^c)\to\Gamma$, there exists an immersed disc $D\looparrowright G^c$ whose framed boundary coincides with the attaching region of $G^c$, such that the double point loops of $D$, considered as fundamental group elements by making some choice of basing path, are mapped to the identity element of $\Gamma$ by $\phi$.
\end{definition}

We will need the following lemma, to trade intersections between distinct surfaces for self-intersections.

\begin{lemma}[Geometric Casson lemma~\cite{F}*{Lemma~3.1} (see also~\cite{DET-book-basicgeo}*{Lemma~15.3}]\label{lem:geometric-casson-lemma}
Let~$F$ and $G$ be transverse generic immersions of compact surfaces in a connected 4-manifold $M$.
Assume that the intersection points $\{p,q\}\subset F\pitchfork G$ are paired by a Whitney disc $W$. Then there is a regular homotopy from $F \cup G$ to $\ol F \cup \ol{G}$ such that  $\ol F \pitchfork \ol{G} = (F\pitchfork G) \smallsetminus \{p,q\}$, that is the two paired intersections have been removed.
The regular homotopy may create many new self-intersections of $F$ and $G$; however, these are algebraically cancelling. Moreover, the regular homotopy is supported in a small neighbourhood of $W$.
\end{lemma}

Applications of this lemma, proven inductively on the number of intersection points, include the following.
 \begin{enumerate}[label=(\roman*)]
   \item Making $F$ and $G$ disjoint, if all intersection points $F\cap G$ are paired by Whitney discs.
   \item Turning algebraically dual spheres $G$ for $F$ into geometrically dual spheres $\ol G$.
 \end{enumerate}

The process of \emph{contraction and push-off}, introduced in~\cite{FQ}*{Section~2.3} (see also~\cite{DET-book-basicgeo}*{Section~15.2.5}) will be important in our proof. \emph{Contraction} converts a capped surface into an immersed disc using two parallel copies of both caps. Given a capped grope, we can iteratively contract caps, to eventually obtain a collection of immersed spheres or discs called the \emph{total contraction}.

After contracting a surface, any other surface that intersected the caps can (but does not necessarily have to) be pushed off the contraction. This reduces the number of intersection points between the resulting contraction and the pushed off surfaces. Suppose that a surface $A$ intersects a cap of the capped surface, and a surface $B$ intersects a dual cap. Then after pushing both $A$ and $B$ off the contraction, we obtain two intersection points between $A$ and $B$. The contraction push-off operation is shown in \cref{fig:contraction-push-off}.

\begin{figure}[htb]
\begin{subfigure}{\textwidth}
	\centering
	\includegraphics[width=\linewidth]{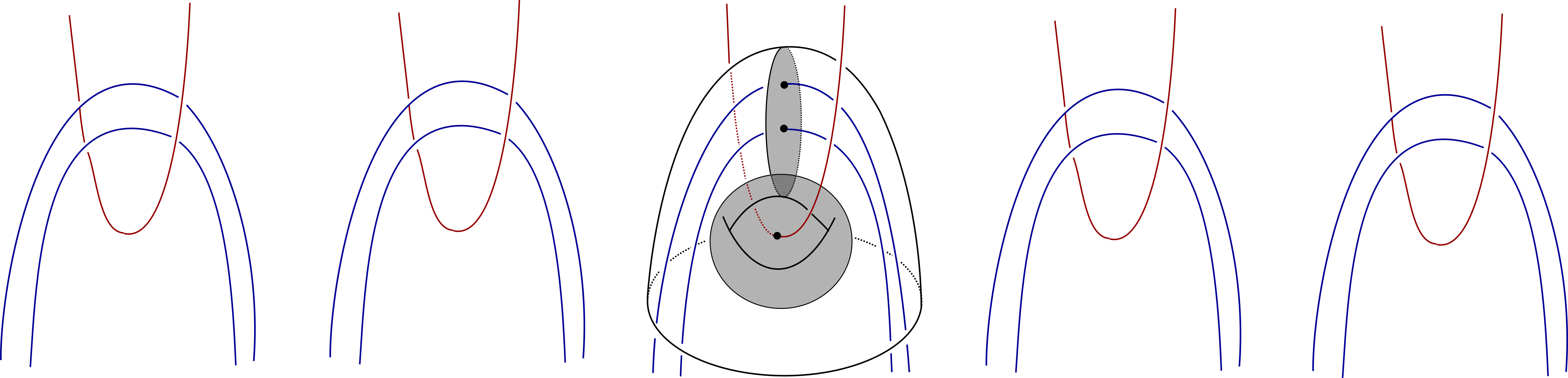}
	\caption{Before contraction of a surface.}
	\label{fig:contraction-push-off-1}
\end{subfigure}
\newline
\vspace{2mm}
\begin{subfigure}{\textwidth}
	\centering
	\includegraphics[width=\linewidth]{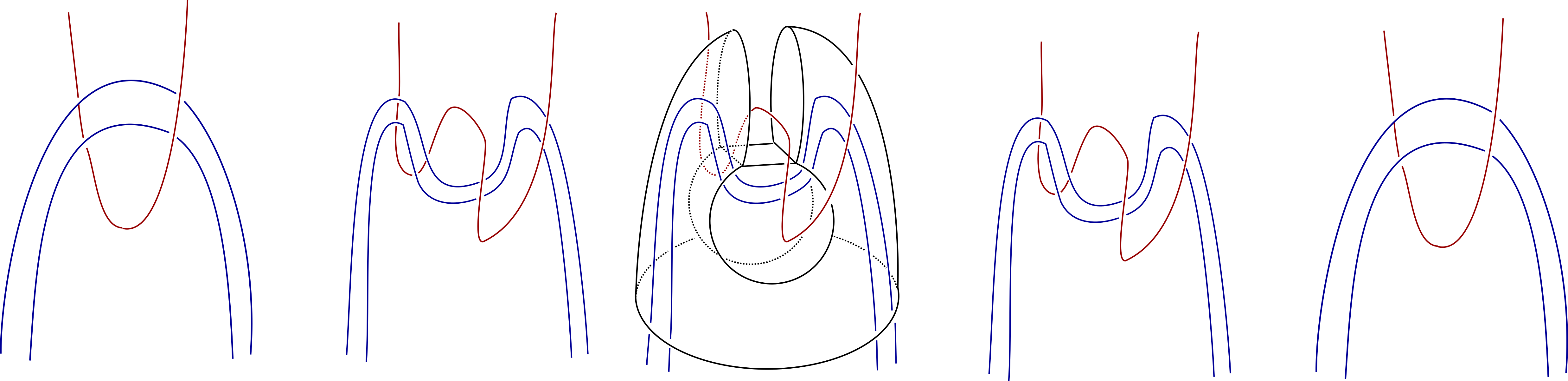}
	\caption{After contraction of a surface, and pushing other surfaces off the caps.}
	\label{fig:contraction-push-off-2}
\end{subfigure}
	\caption{Contraction and push-off. Note the intersections of pushed-off surfaces that occur between diagrams one and two and between diagrams four and five in the bottom row of figures, namely one intersection in the past and one intersection in the future between each pair of surfaces pushed off dual caps. }
	\label{fig:symmetric-contraction2}\label{fig:contraction-push-off}
\end{figure}

\begin{lemma}\label{lem:surgery-indep-caps}
  The homotopy class of the sphere or disc resulting from symmetric contraction of a fixed surface is independent of the choice of caps, provided the boundaries of the different choices of caps coincide.
\end{lemma}

\begin{proof}
  As explained in \cite{FQ}*{Section~2.3}, an isotopy in the model induces a homotopy of the immersed models, so the symmetric contraction is homotopic to the result of surgery along one cap per dual pair. Now let $\{C_i,D_i\}_{i=1}^g$ and $\{C_i',D_i'\}_{i=1}^g$ be two sets of caps for a surface of genus~$g$, such that $\partial C_i =\partial C_i'$ and $\partial D_i = \partial D_i'$ are a dual pair of curves on the surface for each $i$. Then symmetric surgery on $\{C_i,D_i\}_{i=1}^g$ is homotopic to surgery on the $\{C_i\}$, which is homotopic to symmetric surgery on $\{C_i,D_i'\}_{i=1}^g$. This is homotopic to asymmetric surgery on the caps $\{D_i'\}$, which finally is homotopic to the result of symmetric surgery on $\{C_i',D_i'\}_{i=1}^g$, as asserted.
\end{proof}

\section{Constructing capped gropes and towers with dual spheres}\label{sec:technical-lemma}

Our main technical lemma, given below, shows how to upgrade a collection of immersed discs with certain types of dual capped surfaces to a collection of gropes whose attaching regions coincide with the framed boundary of the original discs, as well as the same dual capped surfaces. The requirement on the dual capped surfaces is that the bodies be located close to the boundary of the ambient manifold. This property enables us to find Whitney discs which do not intersect the bodies, at various steps of the argument. The extra power of the dual capped surfaces, compared to dual spheres, is that we can use them to produce dual spheres multiple times in succession.

\begin{lemma}\label{lem:grope-lemma}
Consider a generically immersed collection of discs in a 4-manifold~$M$
\[
D=(D_1,\dots,D_k) \colon (D^2\sqcup \cdots \sqcup D^2,S^1 \sqcup \cdots \sqcup S^1) \looparrowright (M,\partial M).
\]
Suppose that $\{T^c_i\}$ is a dual collection of capped surfaces for the $\{D_i\}$ such that
\begin{equation}\label{lambda-mu-zero}
\lambda(C_\ell, C_m) =\mu(C_\ell)=0
\end{equation}
for every pair of caps $C_\ell$ and $C_m$ of $\{T_i^c\}$. Assume in addition that the body $T_i$ of each $T_i^c$ is contained in a collar neighbourhood of~$\partial M$, and that otherwise only boundary collars of each of the~$D_i$ and the $C_\ell$ lie in the collar neighbourhood.

Then there exists a collection of mutually disjoint capped gropes $\{G_i^c\}$, properly embedded in $M$, with arbitrarily many surface stages and pairwise disjoint caps with algebraically cancelling double points, such that parallel copies of the $\{T^c_i\}$, after a regular homotopy of the caps, provides a dual collection of capped surfaces for the $\{G_i^c\}$. Moreover, for each~$i$, the attaching region of $G_i^c$ coincides with the framed boundary of $D_i$.
\end{lemma}

\begin{proof}
Parameterise the collar neighbourhood of $\partial M$ as $\partial M \times [0,1] \subseteq M$, where $\partial M = \partial M \times \{0\}$. Assume that initially the bodies $\{T_i\}$ of the dual capped surfaces lie in $\partial M \times (5/6,1)$. We will take parallel copies of the $T_i^c$ throughout the proof of this lemma, and also in the proof of \cref{lem:grope-to-tower} below, such that the bodies lie progressively closer to $\partial M \times \{0\}$, in regions of the form $\partial M \times ((6-n)/6,(7-n)/6)$, $n =1,\dots,6$, and caps are obtained as parallel copies of the original caps extended at the boundary collar.

Let $\{C_\ell\}$ denote the caps of the $\{T_i^c\}$. Tube all the intersections and self-intersections among the $\{D_i\}$ into parallel copies of the dual surfaces $\{T^c_i\}$, with bodies $\{T_i\}$ contained in $\partial M \times (5/6,1)$.  This produces a mutually disjoint collection of capped surfaces $\{D'_i\}$ with the same framed boundaries as $\{D_i\}$. Let $\{C'_n\}$ denote the caps for~$\{D'_i\}$.

Now we will separate the collections $\{C_\ell\}$ and $\{C'_n\}$. For each $i$, take a parallel copy of $T_i$ in $\partial M \times (4/6,5/6)$, along with its caps, and contract to obtain a sphere $S_i$. Push nothing off the contraction. The collection $\{S_i\}$ is geometrically dual to $\{D'_i\}$. By construction and \eqref{lambda-mu-zero}, all the intersections between $\{C_\ell\}$ and $\{C'_n\}$ can be paired by framed Whitney discs $\{W_k\}$.
Since $\partial W_k$ lies in $M \sm (\partial M \times [0,4/6))$, we may and shall assume that $\{W_k\}$ does not intersect~$\partial M \times [0,4/6)$;
this follows since each constituent Whitney disc is obtained as a track of a null homotopy,
and because the inclusion map~$M \sm (\partial M \times [0,4/6)) \hookrightarrow M$ is a homotopy equivalence.  From now on we consider parallel copies of $\{T_i^c\}$ with each of the bodies $T_i$ contained in $\partial M \times [0,4/6)$. We continue to denote these parallel copies by $T_i^c$.

By tubing the  $\{W_k\}$ into the $\{S_i\}$ we can also arrange that $\{W_k\}$ does not intersect~$\{D_i'\}$.
Now the geometric Casson lemma (\cref{lem:geometric-casson-lemma}) applied with the Whitney discs $\{W_k\}$ ensures that there are no intersections between the collections $\{C_\ell\}$ and $\{C'_n\}$.
The fact that $W_k \cap D_j' = \emptyset = W_k \cap T_i$, for all $i,j,k$, implies that no unwanted cap--body intersections are created whenever $T_i^c$ is a parallel copy whose body lies in $\partial M \times [0,4/6)$.

Push the intersections and self-intersections of the caps $\{C'_n\}$ down to $\{D'_i\}$ and tube into parallel copies of $\{T_i^c\}$, with bodies in $\partial M \times (3/6,4/6)$, to produce height two capped gropes $\{D_i''\}$. Let $\{C''_p\}$ denote the caps of $\{D''_i\}$. Next we separate the collections $\{C''_p\}$ and $\{C_\ell\}$, as we did above. The procedure is the same: contract parallel copies of $\{T_j^c\}$ with bodies in $\partial M \times (2/6, 3/6)$, pushing nothing off the contraction, to obtain a collection $\{S'_i\}$ of framed spheres geometrically dual to $\{f''_i\}$; find framed Whitney discs for the intersections between $\{C''_p\}$ and $\{C_\ell\}$ which have interiors disjoint from $\partial M \times [0,2/6)$, as well as from $\{D''_i\}$ by tubing into $\{S'_i\}$. Then apply the geometric Casson lemma.  Here, since the boundaries of the Whitney discs lie on $\bigcup \{C_\ell\} \cup \{C''_p\} \subseteq M \sm (\partial M \times [0,2/6))$, we may assume that the Whitney discs created by choosing a null homotopy miss $\partial M \times [0,2/6)$, and hence miss the body $T_i$ of any parallel copy of $T_i^c$ whose body lies in $\partial M \times [0,2/6)$. From now on we will only consider such parallel copies of $T_i^c$.

Apply grope height raising~\cite{FQ}*{Proposition~2.7} (see also~\cite{DET-book-gropeheightraising}*{Proposition~17.1 and Lemma~17.7})
to the capped gropes $\{D_i''\}$, to produce a collection of capped gropes $\{G_i^c\}$ with the same framed attaching region as the $\{D_i''\}$, arbitrarily many surface stages, and pairwise disjoint caps with algebraically cancelling double points. Additionally, parallel copies of the surfaces $\{T_i^c\}$ with bodies in $\partial M \times [0,2/6)$ are geometrically dual to $\{G_i^c\}$. Moreover, the first two stages of each $G_i^c$ coincide with the first two stages of $D_i''^c$, for each $i$.  This completes the proof.
\end{proof}

Next, we show that a collection of capped gropes with certain types of dual capped surfaces, such as those produced by the previous lemma, can be replaced by a collection of $1$-storey capped towers, with the same framed attaching region as the original capped gropes, and with geometrically dual spheres. The following lemma is the only point in this paper that uses the hypothesis that the fundamental group of the ambient manifold  be good. Recall that, roughly speaking, a $1$-storey capped tower can be built from a capped grope, all of whose cap intersections are self-intersections, by adding a second layer of caps to the double point loops of the caps of the grope i.e.\ to the tip regions on the disc stage of the capped grope.

\begin{lemma}\label{lem:grope-to-tower}
Let $M$ be a connected $4$-manifold with $\pi_1(M)$ good. Let $n$ be a non-negative integer. Let $\{G_i^c\}$ be a collection of capped gropes with height $n+2.5$ and mutually disjoint caps, properly embedded in $M$, with a geometrically dual collection of capped surfaces $\{T_i^c\}$, such that the body $T_i$ of each $T_i^c$ is contained in a collar neighbourhood $\partial M \times [0,2/6)$ of~$\partial M$. Suppose that otherwise only boundary collars of the attaching regions of the $G_i^c$ and the caps of the $T_i^c$ lie in $\partial M \times [0,2/6)$.

Then there exists a collection of $1$-storey capped towers $\{\TT_i^c\}$, properly embedded in $M$, where the first storey grope has height $n$, with a geometrically dual collection of  spheres $\{R_i\}$, such that~$\TT_i^c$ and $G_i^c$ have the same attaching region for each $i$. Moreover, the first $n$ surface stages of $G_i^c$ and $\TT_i^c$ coincide and each~$R_i$ is obtained from $T_i^c$ by contraction.
\end{lemma}

\begin{remark}
At first glance \cref{lem:grope-to-tower} seems remarkably close to~\cite{FQ}*{Proposition 3.3}, which claims to begin with a collection of properly immersed discs in a $4$-manifold $M$ equipped with a collection of $\pi_1$-null geometrically transverse capped surfaces, and replace this with $1$-storey capped towers with arbitrary grope height, the same framed attaching region, and equipped with geometrically dual spheres. However, the proof of~\cite{FQ}*{Proposition 3.3} does not construct the claimed geometrically transverse spheres. We also note that the definition of geometrically transverse capped surfaces in~\cite{FQ} allows intersections with the caps, which is different from our notion of dual capped surfaces in \cref{defn:dual-cappped-gropes-surfaces}.
\end{remark}

\begin{proof}
Consider the union of the top $1.5$ stages of the gropes $\{G_i^c\}$. Since $\pi_1(M)$ is good (\cref{def:good-group}) and the caps are mutually disjoint, each component contains an immersed disc whose double point loops are null-homotopic in~$M$, and whose framed boundary coincides with the attaching region of the old top $1.5$ stages. Attach these discs to the lower stages, producing capped gropes $\{\widetilde{G}_i^c\}$ of height~$n+1$. Contract the top stage to obtain capped gropes of height $n$, whose double point loops are still null-homotopic in $M$, with caps still mutually disjoint, and such that the caps have algebraically cancelling self-intersection points, since they arose from a symmetric contraction. Here we are using the fact that the new double point loops are parallel push-offs of the previous double point loops.

Null homotopies for the double point loops produce immersed discs $\{\widetilde{\delta}_\alpha\}$ bounded by the double point loops of the new caps. These new discs may be assumed to miss $\{T_i\}$, since any null homotopy may be pushed off the collar $\partial M \times [0,2/6)$, and since the boundaries of these discs, the double point loops, are disjoint from $\partial M \times [0,2/6)$ because of the hypothesis that $\widetilde{G}_i^c \cap (\partial M \times [0,2/6))$ consists of a boundary collar of the attaching region of $\widetilde{G}_i^c$, for each~$i$.
 On the other hand the discs coming from the null homotopies might intersect $\{\widetilde{G}_i^c\}$ arbitrarily.

Contract parallel copies of $\{T_i^c\}$, with bodies in $\partial M \times (1/6,2/6)$, to produce a family of spheres~$\{S''_i\}$ geometrically dual to~$\{G_i^c\}$. Boundary twist the discs $\{\widetilde{\delta}_\alpha\}$ to achieve the correct framing and then push down and tube into~$\{S''_i\}$ to remove any intersections of the resulting discs with~$\{\widetilde{G}_i^c\}$. Glue the resulting discs~$\{\delta_\alpha\}$ to~$\{\widetilde{G}_i^c\}$ to produce the $1$-storey capped towers $\{\TT_i^c\}$.

Now consider parallel copies of the $\{T_i^c\}$ with bodies in $\partial M \times (0,1/6)$.
At this point the caps of~$\{T_i^c\}$ only (possibly) intersect $\{\TT_i^c\}$ in the tower caps, and the body $T_i$ is geometrically dual to~$\TT_i^c$. Contract each~$T_i^c$ along its caps and call this family of spheres~$\{R_i\}$. Push all intersections with tower caps off the contraction.
The family of dual spheres $\{R_i\}$ is then geometrically dual to the resulting $1$-storey capped towers $\{\mathcal{T}_i^c\}$, as desired.
\end{proof}

We need one more lemma, that we shall use to control the homotopy classes of the geometrically dual spheres in the output of the disc embedding theorem.

\begin{lemma}\label{lem:null-homotopic}
Let $N$ be a 4-manifold and let
  \[D = (D_1,\dots,D_m) \colon (D^2 \sqcup \cdots \sqcup D^2,S^1 \sqcup \cdots \sqcup S^1) \to (N,\partial N)\]
be a generic immersion of a collection of discs that admits a geometrically dual, generically immersed collection $\{E_i\}_{i=1}^m$ of framed spheres in $N$.
Let $T^c \subseteq N \sm \bigcup_{i=1}^m \nu D_i$ be a capped surface constructed by taking a Clifford torus corresponding to an intersection point between $D_i$ and $D_j$, where $i=j$ is permitted, and tubing meridional caps obtained from meridional discs for $D_i$ and $D_j$ into parallel copies of $E_i$ and $E_j$.  Let $S \colon S^2 \to N \sm \bigcup_{i=1}^m \nu D_i$ be the 2-sphere obtained by contracting~$T^c$.  Then $\iota_*[S]=0 \in \pi_2(N)$, where $\iota \colon N \sm \bigcup_{i=1}^m \nu D_i \to N$ is the inclusion map.
\end{lemma}

\begin{proof}
By \cref{lem:surgery-indep-caps}, and as explained in \cite{FQ}*{Section~2.3}, the homotopy class of a 2-sphere obtained by contracting a torus along caps is independent of the choice of caps, provided the boundaries of the different choices of caps coincide.  Therefore in $N$ we may replace the caps constructed from $E_i$ and $E_j$ by the meridional caps.  The sphere $S'$  resulting from contraction along the meridional caps is contained in a $D^4$ neighbourhood in $N$ of the intersection point giving rise to the Clifford torus. So $S'$ is null-homotopic in $N$. It follows that $\iota \circ S$ is null-homotopic. % as desired.
\end{proof}

\section{Proof of the disc embedding theorem}\label{sec:proofs}

In this section we prove \cref{thm:FQ51A-intro}, which we recall below.

\begin{theorem-named}[\cref{thm:FQ51A-intro}\,\,\normalfont{(\cite{FQ}*{Theorem~5.1A})}]
Let $M$ be a connected $4$-manifold with good fundamental group. Consider a continuous map
\[
F=(f_1,\dots,f_k) \colon (D^2\sqcup \cdots \sqcup D^2,S^1 \sqcup \cdots \sqcup S^1) \longrightarrow (M, \partial M)
\]
that is a locally flat embedding on the boundary and that admits algebraically dual spheres $\{g_i\}_{i=1}^k$ satisfying $\lambda(g_i,g_j)=0= \wt{\mu}(g_i)$ for all $i,j$.
Then there exists a locally flat embedding
\[
\ol{F}=(\ol{f}_1,\dots,\ol{f}_k) \colon (D^2\sqcup \cdots \sqcup D^2,S^1 \sqcup \cdots \sqcup S^1) \hookrightarrow (M,\partial M)
\]
such that $\ol{F}$ has the same boundary as $F$ and admits a generically immersed, geometrically dual collection of framed spheres~$\{\ol g_i\}_{i=1}^k$, such that $\ol{g}_i$ is homotopic to $g_i$ for each~$i$.

Moreover, if $f_i$ is a generic immersion, then it induces a framing of the normal bundle of its boundary circle. The embedding $\ol{f}_i$ may be assumed to induce the same framing.
\end{theorem-named}

\begin{proof}
By \cref{thm:gen-immersions-intro}, we may assume that the collections $\{f_i\}$ and $\{g_i\}$ are generically immersed and transverse. By adding local cusps if necessary, assume that $\mu(g_i)=0$ for each $i$. Tube each intersection and self-intersection within $\{f_i\}$ into $\{g_i\}$ using the unpaired intersection point (after having chosen a pairing of all but one of the $f_i$-$g_i$ intersection points by Whitney discs, which is possible since $\lambda(f_i,g_i)=1$). We then obtain a collection of discs, which we still call $\{f_i\}$, where $\lambda(f_i,f_j) = \mu(f_i)=0$ and $\lambda(f_i,g_j)=\delta_{ij}$ for each $i,j$. Note also that the framing on the boundary of each $f_i$ is unchanged since all $g_j$ are framed. Apply the geometric Casson lemma (\cref{lem:geometric-casson-lemma}) to arrange that $\{f_i\}$ and $\{g_i\}$ are geometrically dual, after a regular homotopy.

Since $\lambda(f_i, f_j) = \mu(f_i)=0$ for all $i,j$, the intersections and self-intersections among the $\{f_i\}$ can be paired up with framed Whitney discs $\{D_j\}$. By tubing into the geometric duals $\{g_i\}$ we can assume that the interiors of $\{D_j\}$ lie in the complement of the $\{f_i\}$.

Let $M':= M \setminus \bigcup \nu(f_i)$. We will apply \cref{lem:grope-lemma,lem:grope-to-tower} in $M'$, so we check the hypotheses. The collection $\{f_i\}$ is $\pi_1$-negligible in $M$ so $\pi_1(M')\cong \pi_1(M)$ is good. Let $T_j$ be the Clifford torus at one of the double points paired by $D_j$. Cap each $T_j$ with meridional discs to $\{f_i\}$, then tube the unique intersection point of each cap with $\{f_i\}$ into parallel copies of the dual spheres~$\{g_i\}$. The resulting capped surfaces $\{T_j^c\}$ lie in $M'$ as desired. The bodies lie in a collar neighbourhood of $\partial M'$, and are are geometrically dual to $\{D_j\}$, while the caps have algebraically cancelling intersections. Contract a parallel copy of each $T_j^c$ to produce a dual sphere $S_j$ for $D_j$ in $M'$, and push the collection $\{D_j\}$ off the contraction. Remove the intersections between the caps of $\{T_j^c\}$ and $\{D_j\}$ by tubing into parallel copies of $\{S_j\}$. The resulting caps for $\{T_j^c\}$ no longer intersect $\{D_j\}$.

Now apply \cref{lem:grope-lemma,lem:grope-to-tower} to $M'$, to replace the discs $\{D_j\}$ with $1$-storey capped towers~$\{\TT_j^c\}$ whose framed attaching regions coincide with the framed boundary of $\{D_j\}$ and that have geometrically dual spheres $\{R_j\}$.
Remove intersections of $\{g_i\}$ with~$\{\TT_j^c\}$, by pushing down into the base surface and tubing into~$\{R_j\}$. The resulting spheres $\{\overline{g}_i\}$ are disjoint from $\{\TT_j^c\}$ and geometrically dual to $\{f_i\}$. By~\cite{FQ}*{Chapters~3 and 4} (see also~\citelist{\cite{DET-book-towerheightraising}*{Chapter~17}\cite{Freedman-notes}*{Part~IV}}) every disc-like $1$-storey capped tower with at least four surface stages contains a locally flat embedded disc whose framed boundary coincides with the attaching region of the capped tower. Applied to the collection $\{\TT_j^c\}$, this produces mutually disjoint, embedded and framed Whitney discs pairing the intersections and self-intersections of the $\{f_i\}$ away from $\{\overline{g}_i\}$. Whitney moves guided by these discs produce embedded discs $\{\overline{f}_i\}$ geometrically dual to $\{\overline{g}_i\}$. In the case that $f_i$ was initially generically immersed, we obtain the same framing on the boundary of $\overline{f}_i$ since these are regular homotopies in the interior.

It remains only to argue that each $\overline{g}_i$ is homotopic in $M$ to $g_i$ for each~$i$. Observe that to obtain~$\overline{g}_i$, we have:
\begin{enumerate}[label=(\roman*)]
   \item  homotoped $g_i$ by an application of the geometric Casson lemma, and then
   \item  tubed into parallel copies of the dual spheres $\{R_j\}$ for the towers $\{\TT_j^c\}$.
 \end{enumerate}
However, $R_j$ was obtained by contracting $T^c_j$, whose body is a Clifford torus for an intersection point among the $\{f_i\}$.
Therefore by \cref{lem:null-homotopic}, $R_j$ is null-homotopic in $M$, i.e.\ $[R_j]=0 \in \pi_2(M)$.  (Note that $R_j$ is nontrivial in $\pi_2(M')$.)
It follows that $\overline{g}_i$ is homotopic in $M$ to $g_i$, as desired. This completes the proof.
\end{proof}

\section{Generic immersions in smooth and topological 4-manifolds}\label{sec:generic-immersions}\label{sec:topological}

The goal of this section is to prove \cref{thm:gen-immersions-intro}, rectifying another omission in \cite{FQ}.

\begin{remark}\label{rem:logical-dependencies}
During the proof we will use results that rely on the smooth-input disc embedding theorem. Therefore the results of this section rely on a version of \cref{thm:FQ51A-intro} where $M$ is assumed to be a smooth 4-manifold.  \cref{prop:hom-gen-immersion} was used in the proof of \cref{thm:FQ51A-intro} with a purely topological input, in order to find generic immersions.  The logical dependencies in the development of topological 4-manifold topology are elucidated in \cite{DET-book-flowchart}.
\end{remark}

Recall that for a compact surface $\Sigma$ and $4$-manifold $M$, a {\em generic immersion} in the smooth category, written $f\colon \Sigma^2\imra M^4$, is a smooth map which is an embedding, except for a finite number of transverse double points in the interior. This means that $f$ is an embedding on $\partial \Sigma$ and there are coordinates on $\Sigma$ and on $M$ such that restricted to the interior, $f$ coincides in local coordinates in~$M$  with either the inclusion $\R^2  \times  \{0\}\subset \R^4$ or a transverse double point $\R^2 \times \{0\} \cup \{0\} \times \R^2 \subset \R^4$.

In a smooth $4$-manifold $M$, the set of generic immersions is open and dense in the Whitney topology on $C^\infty(\Sigma,M)$. To see this, apply \cite{Hirsch-Diff-Top}*{Theorem~2.2.12} which shows that immersions are dense in the space of smooth maps, and then~\cite{GoGu}*{Chapter~III, Corollary~3.3} shows that generic immersions are dense in the space of immersions. These are exactly the stable maps in the smooth mapping space~\cite{GoGu}*{Chapter~III, Theorem 3.11}, where a map $f$ is said to be \emph{stable} if it has a neighbourhood in  $C^\infty(\Sigma,M)$ such that every map in the neighbourhood can be obtained by pre-composing with a diffeomorphism of $\Sigma$ and post-composing with a diffeomorphism of $M$, with  both diffeomorphisms  isotopic to the identity.

It is well known that every continuous map between smooth manifolds is arbitrarily close to a smooth map (see for example \cite{Lee-smooth}*{Theorem~10.21} or \cite{Hirsch-Diff-Top}*{Theorem~2.2.6}). Since moreover the collection of generic immersions is open and dense in $C^\infty(\Sigma,M)$, a smooth map can be further perturbed to a (smooth) generic immersion. Since the perturbations may be chosen to be small, a proper continuous map may be perturbed to a proper smooth generic immersion (see~\cite{Spring-proper}*{Lemma~1}).

A continuous map between topological manifolds is called an \emph{immersion} if it is locally an embedding.
A continuous map of a surface to a topological $4$-manifold with the same local behaviour as a smooth generic immersion will be called a \emph{generic immersion} in the topological category. Note that this implies that the map is a \emph{locally flat} embedding near points with a single inverse image.

A smooth homotopy $H$ between smooth maps $\Sigma\rightarrow M$ is said to be generic if the corresponding map $\Sigma\times [0,1]\rightarrow M\times[0,1]$ is a generic smooth map. Whitney's classification of singularities \cites{Whitney-singularities, Whitney-sing-II} of generic maps from 3-manifolds to 5-manifolds implies that the singularities of the track of a  generic homotopy $H$ as above consist of finger moves, Whitney moves and cusps. These arise at finitely many times $t\in I$,  when $H_t\colon \Sigma\to M$ is not a generic immersion but either:
\begin{enumerate}[label=(\roman*)]
\item $H_t$ has a tangency, increasing or decreasing the double point set by a pair with opposite signs, corresponding to a finger/Whitney move, or
\item the rank of the derivative of $H_t$ drops at a single point, creating a cusp where one double point appears or disappears.
 \end{enumerate}
A topological generic homotopy is defined to be a concatenation of finger moves, Whitney moves, and cusps.

A continuous homotopy between smooth maps $\Sigma\to M$ may be perturbed (rel.\ boundary) to produce a smooth generic homotopy. Since the perturbation may be chosen to be small, a proper homotopy may be perturbed to a proper smooth generic homotopy.

Our goal in this section is to state and prove purely topological analogues of these smooth facts. Here is the main technical result.

\begin{proposition}\label{prop:hom-gen-immersion}
Let $\Sigma$ be a compact surface and let $M$ be a $4$-manifold.
\begin{enumerate}
 \item\label{item:hom-to-gen-imm-1}
 Every continuous map $f\colon (\Sigma,\partial\Sigma) \to (M,\partial M)$ is homotopic to a generic immersion, smooth in $M \sm \{q\}$ for some point $q \in M$.
\item\label{item:hom-to-gen-imm-2} Every continuous homotopy $H \colon (\Sigma,\partial\Sigma) \times [0,1] \to (M,\partial M)$ that restricts to a smooth generic immersion on $\Sigma \times \{i\}$ for each $i=0,1$, with respect to some smooth structure on $M$ minus a point $q$, is homotopic rel.\ $\Sigma \times \{0,1\}$  to a generic homotopy, that is smooth in some smooth structure on $M\sm \{r\}$ for some point $r \in M$.
\end{enumerate}
\end{proposition}

\begin{proof}
First we prove~\eqref{item:hom-to-gen-imm-1}, following~\cite{FQ}*{Corollary~9.5C}.
We may assume without loss of generality that $M$ is connected. Choose a smooth structure on $\Sigma$. The complement of a point~$p$ in the interior of $M$ is smoothable relative to any fixed chosen smoothing of the boundary of~$M$~\citelist{\cite{FQ}*{Theorem~8.2}\cite{Quinn-annulus}\cite{Quinn:isotopy}}. Choose a smooth structure on $M\setminus \{p\}$.

If the point $p$ can be chosen disjoint from the image $f(\Sigma)$ in \eqref{item:hom-to-gen-imm-1}, and disjoint from $H(\Sigma \times [0,1]$ in \eqref{item:hom-to-gen-imm-2}, then we have a continuous map, and the result follows from smooth approximation and general position. The aim of the rest of the proof is therefore to arrange that $f$ (respectively $H$ can be arranged to have image missing a point. The proof given in \cite{FQ}*{Corollary~9.5C} assumes that this holds without comment.

Homotope the restriction of $f$ to $\partial \Sigma$ to a smooth embedding, and extend this to a homotopy of $\Sigma$ supported in a collar near the boundary (or use a given embedding to start with and work rel.\ boundary). Consider the smooth surface $\Sigma_p := \Sigma \smallsetminus f^{-1}(p)$, which comes with a proper map~$f|\colon \Sigma_p\to M \smallsetminus \{p\}$ that is properly homotopic to a smooth proper map $f'$. By Sard's theorem, some point $q$ in the interior of $M\smallsetminus \{p\}$ does not lie in the image of~$f'$. Moreover, as $\Sigma$ is compact and the homotopy was proper, $f'$ maps each end of $\Sigma_p$  to $p$ (as  did~$f$). Send all of $f^{-1}(p)$ to $p$,  to extend $f'$ to a continuous map $\Sigma \to M \smallsetminus \{q\}$, homotopic in~$M$ to the original map~$f$. The latter homotopy is produced from the proper homotopy between $f|$ and $f'|$ by mapping every end to $p$.

Now as explained above, the map $f'$ is homotopic to a smooth generic immersion $f''\colon  \Sigma \imra M \smallsetminus \{q\}$.  Add $q$ back in and forget the smooth structure to yield a topological generic immersion, noting that as required $f''$ is smooth in some structure on $M\sm \{q\}$ by construction.

Now, to prove~\eqref{item:hom-to-gen-imm-2}, consider a homotopy $H\colon \Sigma \times [0,1] \to M$ whose restriction $H| \colon \Sigma \times \{i\} \to M$ is a smooth generic immersion in some smooth structure on $M \sm \{q\}$ for some $q\in M$, for each $i=0,1$. We follow a similar strategy as the proof of \eqref{item:hom-to-gen-imm-1}. Use a smoothing of $M$ away from $q$,  and consider
\[(\Sigma\times [0,1])_q := (\Sigma \times [0,1]) \sm H^{-1}(\{q\}) \subseteq M \smallsetminus \{q\}.\]
The proper map $(\Sigma\times [0,1])_q \to M \sm \{q\}$ is properly homotopic rel.\ $\Sigma \times \{0,1\}$ to a proper smooth map $H' \colon (\Sigma \times [0,1])_q \to M\sm \{q\}$, that by Sard's theorem misses at least one point $r \in M$.  Since~$\Sigma \times [0,1]$ is compact and the homotopy was proper, $H'$ maps each end of $H^{-1}(q)$ to $q$.   Extend $H'$ to a continuous map $H'' \colon \Sigma \times [0,1] \to M \sm \{r\}$.  Choose a smooth structure on $M \sm \{r\}$, such that $H| = H''| \colon \Sigma \times \{0,1\} \to M$ is a smooth generic immersion.   To achieve this, start with the original smooth structure on $M \sm \{q\}$ restricted to a neighbourhood of $H''(\Sigma \times \{0,1\}) \subseteq M \sm \{q,r\}$, and extend that structure to a smooth structure on $M \sm \{r\}$.

Now homotope $H''|_{\Sigma \times [0,1]}$ rel.\ $\Sigma \times \{0,1\}$ to a smooth generic homotopy in $M \sm \{r\}$.
Add $r$ back in and forget the smooth structure to yield a topological generic homotopy, noting that as required~$H''$ is smooth in some smooth structure on $M\sm \{r\}$.
\end{proof}

\begin{remark}\label{remark:normal-bundles}
Let us discuss some consequences of \cref{prop:hom-gen-immersion} for normal bundles of generic immersions.
A generic immersion $f\colon  \Sigma \imra M$ has a linear normal bundle $\nu(f)\to\Sigma$, in both the smooth and topological categories \cite{FQ}*{Section~9.3}.  However \cref{prop:hom-gen-immersion} tells us that~$f$ is generically homotopic to a smooth immersion in some smooth structure on $M\sm \{q\}$, so this gives an easier proof that up to homotopy $f$ has a linear normal bundle.
Then $f$ comes with a map $\nu(f) \to M$ of the total space into $M$, which is an embedding away from a finite number of plumbings near the double points of~$f$. In the case that $\Sigma$ has nonempty boundary, assume that $f^{-1}(\partial M) = \partial \Sigma$ and that we are given a fixed framing of the normal bundle restricted to the boundary. We call $f$ {\em framed} if $\nu(f)$ comes with a trivialisation and {\em twisted} if the (relative) Stiefel-Whitney class $w_2(f) := w_2(\nu(f)) \in H^2(\Sigma,\partial \Sigma;\Z/2)$ is nonzero. If $\Sigma$ is oriented, $f|_{\partial\Sigma}$ is framed and $f$ is not twisted, then one can add local cusps to $f$ (corresponding to a non-regular homotopy between generic immersions) until the relative Euler number in $H^2(\Sigma,\partial \Sigma;\Z)$ of $\nu(f)$ vanishes and hence $\nu(f)$ becomes trivial and the framing on~$f$ induces the given framing on $\partial \Sigma$.
\end{remark}

Now we prove \cref{thm:gen-immersions-intro} about homotopy classes $[\Sigma,M]$ of maps $f\colon \Sigma\to M$, when $\Sigma$ is a union of spheres or discs.
We write $\{\Sigma,M\}_{\partial}$ for the subspace of all continuous maps that restrict on $\partial\Sigma$ to locally flat embeddings disjoint from the image of the interior of $\Sigma$, and $[\Sigma,M]_{\partial}$ for the set of homotopy classes of such maps.  In this theorem we do not assume that $f^{-1}(\partial M) = \partial \Sigma$. Choose a local orientation of $M$ at the basepoint and assume that $\Sigma$ comes with a whisker to the basepoint. The proof will use topological transversality, which we state first.

\begin{theorem}[\cites{Quinn-annulus, Quinn-TT} (see also \cite{FQ}*{Section~9.5}]\label{thm:transversality}
Let $\Sigma_1$ and $\Sigma_2$ be locally flat proper submanifolds of a topological $4$-manifold $M$ that are transverse to $\partial M$.  There is an isotopy of $M$, supported in any given neighbourhood of $\Sigma_1 \cap \Sigma_2$, taking $\Sigma_1$ to a submanifold $\Sigma_1'$ transverse to $\Sigma_2$.
\end{theorem}

\begin{reptheorem}{thm:gen-immersions-intro}
Let $\Sigma$ be a disjoint union of discs or spheres, and let $M$ be a $4$-manifold.
  The subspace of generic immersions in the space of all continuous maps
leads to a bijection
\[
\frac{F=\{(f_1,\dots,f_m)\colon \Sigma= \Sigma_1  \sqcup \cdots \sqcup \Sigma_m  \imra M \mid \mu(f_i)_1 = 0,\, i=1,\dots,m\}} { \{\text{isotopies, finger moves, Whitney moves}\} } \longleftrightarrow [\Sigma,M]_{\partial}, \]
where $\mu(f_i)_1 \in\Z$ denotes the signed sum of double points of $f_i$ whose double point loops are trivial in $\pi_1(M)$, and $[\Sigma,M]_{\partial}$ denotes the set of homotopy classes of continuous maps that restrict on~$\partial\Sigma$ to locally flat embeddings disjoint from the image of the interior of $\Sigma$.
Moreover, any such homotopy between generic immersions is homotopic rel.\ $\Sigma \times \{0,1\}$ to a sequence of isotopies, finger moves and Whitney moves.
\end{reptheorem}

\cref{theorem:generic-immersions-bijection} could be rephrased more succinctly, but perhaps less transparently, as the statement that the inclusion of the space of generic immersions into $\{\Sigma,M\}_{\partial}$ is a $1$-connected map.

\begin{proof}
The map is well-defined since any isotopy, finger move, or Whitney move is a homotopy. Note that for each $i$, $\mu(f_i)_1$ can be changed arbitrarily by (non-regular) cusp homotopies.
Therefore by \cref{prop:hom-gen-immersion}\eqref{item:hom-to-gen-imm-1} the map is surjective.

For injectivity, consider generic immersions $F, F'\colon \Sigma\looparrowright M$, which restrict to embeddings on $\partial \Sigma$ that miss the image of the interior, and assume that~$F$ and~$F'$ are homotopic. Using topological transversality, we can assume that $F$ and $F'$ intersect transversely after performing an isotopy. Choose a point $q\in M$ disjoint from the image of both $F$ and $F'$. Since $F$ and $F'$ are generic immersions intersecting transversely, there is a neighbourhood, namely the union of the images of the normal bundles, within which $F$ and $F'$ are smooth generic immersions. Extend this smooth structure over $M\smallsetminus \{q\}$.

Now by \cref{prop:hom-gen-immersion}\eqref{item:hom-to-gen-imm-2}, we can replace the homotopy from $F$ to $F'$ by a smooth generic homotopy, by performing a homotopy rel.\ $\Sigma \times \{0,1\}$.  Then we saw earlier that
the singularities of the track of this homotopy consists of finger moves, Whitney moves and cusps.
If $\mu(f_i)_1 = \mu(f'_i)_1$ for every~$i$ then the cusps arising in $H$ can be cancelled in pairs~\cite{FQ}*{p.~23}, leading to a regular homotopy,  which as desired is a sequence of isotopies, finger moves, and Whitney moves.
\end{proof}

\bibliographystyle{halpha-abbrv}
\def\MR#1{}
\bibliography{bib}
\end{document}